\documentclass[10pt]{article}
\usepackage{amsmath,amssymb,amsfonts,amsthm}
\usepackage{graphicx,color,subfigure,enumerate,multirow,longtable}
\usepackage[]{hyperref}
\usepackage{fullpage}

\newcommand\RR{\ensuremath{\mathbb{R}}}

\newcommand\ZZ{\ensuremath{\mathbb{Z}}}

\newtheorem{thm}{Theorem}
\newtheorem{prop}[thm]{Proposition}
\newtheorem{cor}[thm]{Corollary}
\newtheorem{lem}[thm]{Lemma}

\newtheorem{rem}[thm]{Remark}

\title{Courant-sharp eigenvalues of the three-dimensional square torus}
\author{Corentin L\'ena\footnote{Department of Mathematics \emph{Guiseppe Peano}, University of Turin, Via Carlo Alberto, 10, 10123 Turin, Italy 
\texttt{clena@unito.it}}}

\begin{document}
\maketitle
\begin{abstract}
In this paper, we determine, in the case of the Laplacian on the flat three-dimensional torus $(\mathbb{R}/\mathbb{Z})^3\,$, all the eigenvalues 
having an eigenfunction which satisfies the Courant nodal domains theorem with equality (Courant-sharp situation). Following the strategy of {\AA}.~Pleijel (1956), the proof is a combination of  an explicit lower bound of the counting function and a Faber--Krahn-type inequality for domains on the torus, deduced as, in the work of P.~B\'erard and D.~Meyer (1982), from an isoperimetric inequality. This inequality relies on the work of L.~Hauswirth, J.~Perez,  P.~Romon, and A.~Ros (2004) on the periodic isoperimetric problem.
\end{abstract}



\section{Introduction}

A well-known result by R.~Courant in \cite{Cou23Nodal} (see also \cite{CouHil53,Ple56}) gives an upper bound on
the number of nodal domains of an eigenfunctions of the Laplacian. If $\Omega \subset \RR^n$ is an  open, bounded, and connected domain, with a sufficiently regular boundary, if $u$ is an eigenfunction of the Laplacian with a Dirichlet or Neumann boundary condition, ²associated with the $k$-th eigenvalue $\lambda_k(\Omega)\,$, the eigenvalues being arranged in non-decreasing order and counted with multiplicities, then $u$ has at most $k$ nodal domains. In \cite{Ple56}, {\AA}.~Pleijel sharpened this result by showing that, for a given domain $\Omega$ in $\RR^2$, and for the Dirichlet boundary condition, an eigenfunction associated with $\lambda_k(\Omega)$ has less than $k$ nodal domains, except for a finite number of indices $k\,$. This  was generalized in \cite{BerMey82} by P.~B\'erard and D.~Meyer to the case of a compact Riemannian manifold, with or without boundary, with the Dirichlet condition on the boundary, in any dimension. It has been shown by I. Polterovich in \cite{Pol09}, using estimates from \cite{TotZel09}, that the analogous result also holds in the case of the Neumann boundary condition, for a domain in $\RR^2$ with a piecewise-analytic boundary.

These results leave open the question of determining, for a specific domain or manifold, all the cases of equality. This problem has been the object of much attention in recent years, motivated in part by its connection with a minimal partition problem (see \cite{HelHofTer09}). It is stated in \cite{Ple56} that when $\Omega$ is a square, equality can only occur for eigenfunctions having one, two or four nodal domains, associated with the first, the second or the third (which are equal), or the fourth eigenvalue respectively. A complete proof of this fact is given by P. B\'erard and B. Helffer in   \cite{BerHel15Square}. The case of the disk with the Dirichlet boundary condition is treated in \cite{HelHofTer09} and the case of the sphere in \cite{Ley96,HelHofTer10a}. The case of the flat torus $(\RR/\ZZ)^2$ has been studied in \cite{Len15Torus}. The present work is a continuation of this paper in the three-dimensional case. The cases of the equilateral torus, and  of some triangles with a Dirichlet boundary condition are investigated in \cite{BerardHelffer2015Triangle}. The case of the square and the right-angled isosceles triangle with the Neumann boundary condition are treated in \cite{HelfferPersson2015SquareN} and \cite{BanBerFaj2015} respectively. A first three-dimensional example was studied by B. Helffer and R. Kiwan in \cite{HelKiw15Cube}: the cube with a Dirichlet boundary condition. In \cite{HelPer15Balls}, all the cases of equality for the sphere, and for the ball with a Dirichlet or Neumann boundary condition, are determined, in dimension greater than $2\,$. In the present paper, we will show that for the flat torus
$(\mathbb{R}/\mathbb{Z})^3\,$, equality in the Courant nodal domain theorem holds only for eigenfunctions having one or two nodal domains, respectively associated with the first eigenvalue, or eigenvalues two to seven (which are equal). This provides us with another three-dimensional example. 

Let us fix some definitions and notation that will be used in the sequel. In the rest of this paper, $\mathbb{T}^3$ stands for the three-dimensional torus
\[\mathbb{T}^3=(\mathbb{R}/\mathbb{Z})^3\]
equipped with the standard flat metric, and $-\Delta_{\mathbb{T}^3}$ stands for the (non-negative) Laplace-Beltrami operator on $\mathbb{T}^3\,$. If $\Omega$ is an open set in $\mathbb{T}^3$ with a sufficiently regular boundary, we write $(\lambda_k(\Omega))_{k\ge 1}$ for the eigenvalues of $-\Delta_{\mathbb{T}^3}$ in $\Omega$ with the Dirichlet boundary condition on $\partial \Omega\,$, arranged in non-decreasing order and counted with multiplicity. In particular, $\lambda_k(\mathbb{T}^3)$ is the $k$-th eigenvalue of $-\Delta_{\mathbb{T}^3}$. If $u$ is an eigenfunction of $-\Delta_{\mathbb{T}^3}\,$, we call \emph{nodal domains} of $u$ the connected components of $\mathbb{T}^3\setminus u^{-1}(\{0\})\,$, and we denote by $\nu(u)$  the cardinal of the set of nodal domains.
With any eigenvalue $\lambda$ of $-\Delta_{\mathbb{T}^3}\,$, we associate the integer
\[\kappa(\lambda)=\min\{k\in \mathbb{N}^{*} \,:\,\lambda_k(\mathbb{T}^3)=\lambda \}\,.\]
Let us state the Courant theorem in this notation.
\begin{thm}
	\label{thmCourant}
	For any eigenvalue $\lambda$ of $-\Delta_{\mathbb{T}^3}$ and any eigenfunction $u$ associated with $\lambda$, $\nu(u)\le \kappa(\lambda)\,$.
\end{thm}
Following \cite{HelHofTer09}, we say that an eigenvalue $\lambda$ of $-\Delta_{\mathbb{T}^3}$ is \emph{Courant-sharp}, if there exists an associated eigenfunction $u$ such that $\nu(u)=\kappa(\lambda)\,$, that is to say if it satisfies the case of equality in the Courant theorem. We will prove the following result.

\begin{thm} \label{thmPlejielTorus3D} The only Courant-sharp eigenvalues of $-\Delta_{\mathbb{T}^3}$ are $\lambda_k(\mathbb{T}^3)$ with $k\in\{1,2,3,4,5,6,7\}\,$.
\end{thm}

The proof follows the approach used  by \AA. Pleijel in \cite{Ple56} and in the case of a compact manifold by P. B\'erard et D. Meyer in \cite{BerMey82} (see also \cite{Ber82}). In Section \ref{secIso}, we establish an isoperimetric inequality and we use it to prove a Faber--Krahn-type inequality for domains in $\mathbb{T}^3\,$. This is the most delicate part, since the isoperimetric problem in flat tori has not been solved in full generality in dimension greater than two (see \cite{HowHutMor99} for the case of dimension two). To bypass this obstruction, we combine a partial result obtained by L.~Hauswirth, J.~Perez,  P.~Romon, and A.~Ros in \cite{HauPerRomRos04} with a procedure inspired by the method of P.~B{\'e}rard and D.~Meyer in \cite[Appendix C]{BerMey82} (see also \cite[II]{Ber82}). In Section \ref{secLower}, we get a lower bound on the counting function by an elementary counting argument. In Section \ref{secReduc}, we combine these results to show that eigenvalues whose index is greater than $270$ cannot be Courant-sharp, and we identify a (small) finite set containing all Courant-sharp eigenvalues. In Section \ref{secCourantSym}, we use a Courant-type theorem with symmetry to show that one of the eigenvalues in the previous set is not Courant-sharp. This idea goes back to the work of Leydold (see \cite{Ley96}) and was already used, with a similar aim, in \cite{HelHofTer10a,HelfferPersson2015SquareN,HelKiw15Cube}. The only remaining eigenvalues are those of Theorem \ref{thmPlejielTorus3D}. 

\section{A Faber--Krahn-type inequality}
\label{secIso}

Let us first prove the following version of the isoperimetric inequality. In the rest of this section, $|\cdot|$ stands either for the three-dimensional Lebesgue measure or the two-dimensional Hausdorff measure, and $B^3$ for the (Euclidean) unit ball in $\mathbb{R}^3\,$.

\begin{prop} \label{propIneqIsoTorus} Let $\Omega$ be an open set in $\mathbb {T}^3$ with $\left|\Omega\right|\le \frac{4\pi}{81}\,$. We have 
\begin{equation}
\label{eqIneqIsoTorus}
\left|\partial B^3\right|\left|B^3\right|^{-\frac{2}{3}}\le \left(\left|\partial \Omega\right|+2\left| \Omega\right|\right)|\left|\Omega\right|^{-\frac{2}{3}}\,.
\end{equation}
\end{prop} 
We introduce the notation $IR_3:=\left|\partial B^3\right|\left|B^3\right|^{-\frac{2}{3}}$ (this is the minimal isoperimetric ratio in three dimensions, in the Euclidean case). We have $IR_3=(36\pi)^{\frac{1}{3}}\,.$

\begin{rem}
	The right-hand side of Inequality \eqref{eqIneqIsoTorus} is not homogeneous. When $|\Omega|$ becomes small, Inequality \eqref{eqIneqIsoTorus} gets closer to the isoperimetric inequality in the Euclidean case, which is asymptotically optimal according to \cite[Lemma II.15]{BerMey82}.\end{rem}

The proof  relies on the following result, which is a special case of \cite[Theorem 18]{HauPerRomRos04}.

\begin{prop} \label{propIneqIso2Per} Let $\mathcal{U}$ be an open set in $\mathbb{T}^2\times \mathbb{R}$ with $\left|\mathcal{U}\right|\le \frac{4\pi}{81}\,$. Then
\begin{equation*}
	\left|\partial B^3\right|\left|B^3\right|^{-\frac{2}{3}}\le \left|\partial \mathcal{U}\right|\left|\mathcal{U}\right|^{-\frac{2}{3}}\,.
\end{equation*}
\end{prop}

Let us comment on the value $\frac{4\pi}{81}$ appearing in Proposition \ref{propIneqIso2Per}. Following \cite{Ros05}, for any positive number $V>0\,$, we call \emph{isoperimetric region} of volume $V$ an open set $\Omega$ in $\mathbb{T}^2\times \mathbb{R}\,$, with $|\Omega|=V$, such that $|\partial \Omega|$ is minimal. We call \emph{isoperimetric surface} the boundary of an isoperimetric region. 
We also define the \emph{isoperimetric profile} $I$ by
\begin{equation*}
	I(V):=\inf\left\{|\partial \Omega|\,:\,\Omega \subset \mathbb{T}^2\times \RR \mbox{ and } |\Omega|=V\right\}\,.
\end{equation*}  
It is conjectured (see for instance \cite[Section 4]{HauPerRomRos04}) that, depending on the volume of the isoperimetric domain that they bound, isoperimetic surfaces in $\mathbb{T}^2\times \RR$ are either balls, cylinders, or pairs of parallel two-dimensional flat tori, of the form $\mathbb{T}^2\times\{t\}$ with $t\in \RR$. More explicitly, we can define, as in \cite[Section 4]{HauPerRomRos04}, the \emph{spheres-cylinders-planes} profile $I_{SCP}$: for $V>0\,$, $I_{SCP}(V)$ is the least possible area for the boundary of a region of volume $V$, among spheres, cylinders, or pairs of parallel flat tori. Computation shows that
\begin{equation*}
	I_{SCP}(V)=\left\{\begin{array}{lll}
						(36\pi V^2)^{1/3}& \mbox{ if } 0<V\le \frac{4\pi}{81}& \mbox{ (spherical range)};\\
						(4\pi V)^{1/2}& \mbox{ if } \frac{4\pi}{81} \le V \le \frac{1}{\pi}& \mbox{ (cylindrical range)};\\
						2 & \mbox{ if } \frac{1}{\pi} \le V &\mbox{ (planar range)}.
					  \end{array}\right.
\end{equation*} 
The above conjecture can be reformulated as $I(V)=I_{SCP}(V)$ for all $V>0\,$. The result \cite[Theorem 18]{HauPerRomRos04} asserts that $I(V)=I_{SCP}(V)$ for $V \in \left(0,\frac{4\pi}{81}\right]\,$, that is to say that spheres are indeed isoperimetric surfaces in the spherical range (the result applies in fact to more general tori than $\mathbb{T}^2$).

Let us now prove Proposition \ref{propIneqIsoTorus}. We consider the canonical coordinates $(x,y,z)$ on $\mathbb{T}^3= \left(\mathbb{R}/\mathbb{Z}\right)^3$. 
For $t\in [0,1)$ we consider the surface $\mathcal{H}_{z=t}$ in $\mathbb{T}^3$ defined by
\[\mathcal{H}_{z=t}=\left\{(x,y,z)\in \mathbb{T}^3\,:\, z=t\right\}.\] 
According to Fubini's theorem
\[\left|\Omega\right|=\int_0^1 \left|\Omega\cap \mathcal{H}_{z=t}\right|\,dt\,.\]
There exists therefore $t_z\in [0,1)$ such that $\left|\Omega\cap \mathcal{H}_{z=t_z}\right|\le \left|\Omega\right|\,$. Let us now consider the open set $\widetilde{\Omega}$ in $\mathbb{T}^3$ defined by 
$\widetilde{\Omega}:=\Omega \setminus \mathcal{H}_{x=t_z}\,.$ It can be considered as a subset of $\mathbb{T}^2\times \RR\,$. More precisely, let us consider the canonical projection
\[\begin{array}{cccc}
		\Pi:& \mathbb{T}^2\times\RR &\to& \mathbb{T}^3\\
					 & (x,y,z)&\mapsto& (x ,y ,z \mbox{ mod } 1)\,.\\
	\end{array}\]
The set $\Pi^{-1}\left(\mathcal{H}_{z=t_z}\right)$ is constituted of a family of parallel two-dimensional flat tori in  $\mathbb{T}^2\times\RR\,$, separated by a distance of $1\,$. They define a partition of $\mathbb{T}^2\times\RR$ into cells of volume $1\,$. Let us denote by $\widetilde{\Omega}_0$ the intersection of one of these cells with $\Pi^{-1}(\Omega)\,$. It has the same volume as $\widetilde{\Omega}\,$, and its boundary has the same area. Applying Proposition \ref{propIneqIso2Per}, we get
\begin{equation}
\label{eqIneqIsoR3}
\left|\partial B^3\right|\left|B^3\right|^{-\frac{2}{3}}\le \left|\partial \widetilde{\Omega}_0\right|\left|\widetilde{\Omega}_0\right|^{-\frac{2}{3}}\,.
\end{equation}
On the other hand, cutting $\Omega$ with the plane $\mathcal{H}_{z=t_z}$ adds $2\left|\mathcal{H}_{z=t_z}\cap \Omega\right|$ to the area of the boundary.  We therefore get
\begin{equation}
\label{eqIneqPerCut}
	\left|\partial \widetilde{\Omega}_0\right|=\left|\partial \widetilde{\Omega}\right|= \left|\partial \Omega\right|+2\left|\mathcal{H}_{z=t_z}\cap \Omega\right| \le \left|\partial \Omega\right|+2\left|\Omega\right|\,.
\end{equation}
Combining Inequalities \eqref{eqIneqIsoR3} and \eqref{eqIneqPerCut}, we get Inequality \eqref{eqIneqIsoTorus}. We have proved Proposition \ref{propIneqIsoTorus}. The above idea of cutting the domain was used in \cite[Appendix C]{BerMey82}, with geodesic balls, to prove an asymptotic isoperimetric inequality for domains in a Riemannian manifold (see \cite[Lemma II.15]{BerMey82}).

Let us now use Inequality \eqref{eqIneqIsoTorus} to obtain a Faber--Krahn-type inequality.
\begin{prop} 
\label{propFaberKrahn}
If $\Omega$ is an open set in $\mathbb{T}^3$ with $\left|\Omega\right|\le \frac{4\pi}{81}\,$, we have
\begin{equation} 
\label{eqIneqFaberKrahnTorus}
\left(1-\left(\frac{2\left|\Omega\right|}{9\pi}\right)^{\frac{1}{3}}\right)^2\lambda_1(B^3)\left|B^3\right|^{\frac{2}{3}}\le \lambda_1(\Omega)\left|\Omega\right|^{\frac{2}{3}}\,,
\end{equation}
where $\lambda_1(B^3)$ is the first eigenvalue of $-\Delta$ in $B^3$ with the Dirichlet condition on $\partial B^3\,$.
\end{prop}

\begin{rem} 
\label{remRatioFaberKrahn}
We have $\lambda_1(B^3)=\pi^2$ and therefore $\lambda_1(B^3)\left|B^3\right|^{\frac{2}{3}}=\left(\frac{4}{3}\right)^{\frac{2}{3}}\pi^{\frac{8}{3}}\,$.
\end{rem}

\begin{proof} For any open set $\omega$ with $|\omega|\le |\Omega|\,$, we have, thanks to Inequality \eqref{eqIneqIsoTorus},
\begin{equation}
	\label{eqIneqIsoSmall}
	\left(1-\left(\frac{2\left|\Omega\right|}{9\pi}\right)^{\frac{1}{3}}\right)IR_3 \le \left(1-\left(\frac{2\left|\omega\right|}{9\pi}\right)^{\frac{1}{3}}\right)IR_3\le \left|\partial \omega \right|\left|\omega\right|^{-\frac{2}{3}}\,.
\end{equation}
Let us now consider a positive eigenfunction $u$ of $-\Delta_{\mathbb{T}^3}$ on $\Omega$, with Dirichlet boundary condition. For any $t>0$, the level set $\Omega_t:=\{p \in \Omega\,:\, u(p)>t\}$ satisfies $|\Omega_t|\le |\Omega|\,$, and therefore Inequality \eqref{eqIneqIsoSmall} applies with $\omega=\Omega_t\,$. The classical proof of the Faber--Krahn inequality using the co-area formula and the symmetrization of the level sets (see for instance  \cite[I.9]{BerMey82}, or \cite[III.3]{Cha01}), combined with Inequality \eqref{eqIneqIsoSmall}, gives us Inequality \eqref{eqIneqFaberKrahnTorus}.
\end{proof}

\section{Lower bound on the counting function}
\label{secLower}

We define the counting function by  
\[N(\lambda):=\sharp\{k\,:\,\lambda_k(\mathbb{T}^3) <\lambda\}\,.\]

\begin{prop}
\label{propIneqN}
For $\lambda >0$, 
\begin{equation}
\label{eqIneqN}
N(\lambda)\ge   \frac{4\pi}{3}\left(\frac{\sqrt{\lambda}}{2\pi}-\frac{\sqrt{3}}{2}\right)^3\,.
\end{equation}
\end{prop}
\begin{proof} 	The eigenvalues of $-\Delta_{\mathbb{T}^3}$ are of  the form 
	\[\lambda_{m,n,p}=4\pi^2(m^2+n^2+p^2)\,,\]
	with $(m,n,p)\in \mathbb{N}_0^3\,$. With each integer triple $(m,n,p)$ we associate a finite dimensional vector space $E_{m,n,p}$ of eigenfunctions such that 	\[L^2(\mathbb{T}^3)=\overline{\bigoplus_{(m,n,p)}E_{m,n,p}}\,.\] 
The vector space $E_{m,n,p}$ is generated by basis functions of the form
\[(x,y,z)\mapsto \varphi(2m\pi x)\psi(2 n \pi y)\chi(2 p \pi z)\,,\]
where $\varphi$, $\psi$, and $\chi$ are sines or cosines (more precisely, it is generated by the functions of this form which are non-zero). It has dimension $2^e\,$, where $e$ is the number of non-zero integers in the triple $(m,n,p)\,$. For all $\lambda>0$ and $e \in\{0,\,1,\,2,3\,\}$, let us denote by $n_{e}(\lambda)$ the number of integer triples $(m,n,p)\in \mathbb{N}^3\,$ having $e$ non-zero components and satisfying $4\pi^2(n^2+m^2+n^2)<\lambda\,$. Taking the dimension of the spaces $E_{m,n,p}$ into account, we have
\[N(\lambda)=n_0(\lambda)+2n_1(\lambda)+4n_2(\lambda)+8n_3(\lambda).\]
Let us now denote by $B_{\lambda}$ the open ball in $\mathbb{R}^3$ of center $0$ and radius $\sqrt{\lambda}/2\pi\,$. We define
\[n(\lambda):=\sharp\left(\mathbb{Z}^3\cap B_{\lambda}\right)\,.\]
It is easy to see, by taking into account all the possible sign patterns, that
\[n(\lambda)=n_0(\lambda)+2n_1(\lambda)+4n_2(\lambda)+8n_3(\lambda)=N(\lambda)\,.\]
Let us now obtain a lower bound of $n(\lambda)\,$. With  each point $(m,n,p)$ in $\mathbb{Z}^3\cap B_{\lambda}$, we associate the cube 
\[\mathcal{C}_{m,n,p}=\left[m-\frac12,m+\frac12\right]\times \left[n-\frac12,n+\frac12\right] \times \left[p-\frac12,p+\frac12\right]\,.\] 
We have
\begin{equation}
\label{eqCounting}
 n(\lambda)=\left|\bigcup_{(m,n,p)\in \mathbb{Z}^3\cap  B_{\lambda}}C_{m,n,p}\right|\,.
 \end{equation}
Let us now show that we have
\begin{equation}
\label{eqInclusion}
B\left(0,\frac{\sqrt{\lambda}}{2\pi}-\frac{\sqrt{3}}{2}\right)\subset\bigcup_{(m,n,p)\in \mathbb{Z}^3\cap  B_{\lambda}}C_{m,n,p}\,,
\end{equation}
where $B\left(0,\frac{\sqrt{\lambda}}{2\pi}-\frac{\sqrt{3}}{2}\right)$ is the ball in $\RR^3$ of center $0$ and radius ${\sqrt{\lambda}}/{2\pi}-{\sqrt{3}}/{2}\,$. Indeed, let $(x,y,y)\in B\left(0,\frac{\sqrt{\lambda}}{2\pi}-\frac{\sqrt{3}}{2}\right)$, and let $m$, $n$, and $p$ be the closest integers to $x$, $y$ and $z$ respectively. On the one hand
\[(x,y,z)\in C_{m,n,p}\,,\]
and on the other hand, by the triangle inequality,
\[\sqrt{n^2+n^2+p^2}\le \sqrt{x^2+y^2+z^2}+\frac{\sqrt3}{2}<\frac{\sqrt{\lambda}}{2\pi}\,,\]
so that $(m,n,p)\in \ZZ^3\cap B_{\lambda}\,$. Inclusion \eqref{eqInclusion} together with Equality \eqref{eqCounting} gives us Inequality \eqref{eqIneqN}.
\end{proof}

\begin{rem}
	The equality $N(\lambda)=n(\lambda)$ can be proved more directly, by considering the complex Hilbert space of complex-valued $L^2$-functions on $\mathbb{T}^3$, and by choosing the functions $(x,y,z)\mapsto e^{i2 m\pi x}e^{i2 n\pi y}e^{i2 p\pi z}\,$, with $(m,n,p)\in \ZZ^3\,$, as an orthogonal basis of eigenfunctions of $-\Delta_{\mathbb{T}^3}\,$.
\end{rem}

\section{Reduction to a finite set of eigenvalues}
\label{secReduc}

We now proceed with the proof of Theorem \ref{thmPlejielTorus3D}.
\begin{lem} 
\label{lemPleijel}
If $\lambda$ is an eigenvalue of  $-\Delta_{\mathbb{T}^3}$ that has an associated eigenfunction $u$ with $k$ nodal domains, $k\ge 7\,$, then
\begin{equation} 
\label{eqIneqNodalDom}
\left(1-\left(\frac{2}{9\pi k}\right)^{\frac{1}{3}}\right)^2\left(\frac{4\pi^4k}{3}\right)^{\frac{2}{3}}\le \lambda\,.
\end{equation}
\end{lem}
\begin{proof} Since $\left|\mathbb{T}^3\right|=1\,$, one of the nodal domains of $u$ has an area no larger than $\frac{1}{k}\,$.  Let us denote this nodal domain by $D\,$. We have $|D| \le \frac{1}{7} < \frac{4\pi}{81}\,$.  According to Proposition \ref{propFaberKrahn},
\[\lambda=\lambda_1(D)\ge |D|^{-\frac{2}{3}}\left(1-\left(\frac{2\left|D\right|}{9\pi}\right)^{\frac{1}{3}}\right)^2\lambda_1(B^3)\left|B^3\right|^{\frac{2}{3}} \ge k^{\frac{2}{3}}\left(1-\left(\frac{2}{9\pi k}\right)^{\frac{1}{3}}\right)^2\lambda_1(B^3)\left|B^3\right|^{\frac{2}{3}}.\]
Using Remark \ref{remRatioFaberKrahn}, we get Inequality \eqref{eqIneqNodalDom}.
\end{proof}

\begin{cor} If $\lambda$ is a Courant-sharp eigenvalue, $\kappa(\lambda) \le 269\,$.
\end{cor}

\begin{proof} If $\lambda$ is an eigenvalue, $\kappa(\lambda)=N(\lambda)+1>N(\lambda)\,$. From Proposition \ref{propIneqN}, we obtain 
\[\frac{\sqrt{\lambda}}{2\pi}<\left(\frac{3}{4\pi}\right)^{\frac{1}{3}}\kappa(\lambda)^{\frac{1}{3}}+\frac{\sqrt{3}}{2}é\,,\]
while, if $\lambda$ is a Courant-sharp eigenvalue (with $\kappa(\lambda)\ge 7$), Lemma \ref{lemPleijel} implies
\[\frac{\sqrt{\lambda}}{2\pi}\ge \left(\frac{\pi}{6}\right)^{\frac{1}{3}}\kappa(\lambda)^{\frac{1}{3}}-\frac{1}{3}\,.\]
We conclude that $\lambda$ cannot be Courant-sharp if
\[\kappa(\lambda)\ge \left(\frac{\frac{\sqrt{3}}{2}+\frac{1}{3}}{\left(\frac{\pi}{6}\right)^{\frac{1}{3}}-\left(\frac{3}{4\pi}\right)^{\frac{1}{3}}}\right)^3\simeq 269.65\,.\] \end{proof}

\begin{cor}
	\label{corCourantSharp}
	If $\lambda$ is a Courant-sharp eigenvalue of $-\Delta_{\mathbb{T}^3}\,$ with $\kappa(\lambda)\ge 7\,$, then 
	\begin{equation}
	\label{eqIneqCourantSharp}
	\kappa(\lambda)\le \left(\left(\frac{3}{4\pi^4}\right)^{\frac{1}{3}}\sqrt{\lambda}+\left(\frac{2}{9\pi}\right)^{\frac{1}{3}}\right)^3\,.
	\end{equation}
\end{cor}

\begin{cor}
	If $\lambda$ is a Courant-sharp eigenvalue, $\frac{\lambda}{4\pi^2} \in \{0,1,2\}\,$.
\end{cor}
\begin{proof}
	Table \ref{tabEigVal} gives the first $305$ eigenvalues of  $-\Delta_{\mathbb{T}^3}\,$. In this table
	\[K(\lambda):=\left(\left(\frac{3}{4\pi^4}\right)^{\frac{1}{3}}\sqrt{\lambda}+\left(\frac{2}{9\pi}\right)^{\frac{1}{3}}\right)^3\,.\]
	The quantity $K(\lambda)$ is not given for $\frac{\lambda}{4\pi^2} \in \{0,1\}\,$, since $\kappa(\lambda)\le 6$ in those cases and Corollary \ref{corCourantSharp} does not apply. The table shows that if $\frac{\lambda}{4\pi^2}\ge 3\,$, $\kappa(\lambda)>K(\lambda)\,$, and therefore, according to Corollary \ref{corCourantSharp}, $\lambda$ is not Courant-sharp.
\end{proof}

\begin{table}
	\centering
	\caption{The first $305$ eigenvalues\label{tabEigVal}}
	\begin{tabular}{|c|c|c|c|}
	\hline  
	$\frac{\lambda}{4\pi^2}$&multiplicity&$\kappa(\lambda)$&$K(\lambda)$\\ 
	\hline 
    $0$ &     $1$ &     $1$ &      \\  
    $1$ &     $6$ &     $2$ &      \\  
    $2$ &     $12$ &    $8$ &      $10.19$\\  
    $3$ &     $8$ &     $20$ &     $16.83$\\  
    $4$ &     $6$ &     $28$ &     $24.26$\\  
    $5$ &     $24$ &    $34$ &     $32.40$\\  
    $6$ &     $24$ &    $58$ &     $41.16$\\   
    $8$ &     $12$ &    $82$ &     $60.37$\\  
    $9$ &     $30$ &    $94$ &     $70.74$\\   
    $10$ &     $24$ &    $124$ &   $81.58$\\  
    $11$ &    $24$ &    $148$ &    $92.87$\\  
    $12$ &    $8$ &     $172$ &    $104.59$\\  
    $13$ &    $24$ &    $180$ &    $116.71$\\  
    $14$ &    $48$ &    $204$ &   $129.24$\\  
    $16$ &    $6$  &    $252$ &     $155.41$ \\ 
    $17$ &    $48$ &    $258$ &    $169.03$\\
    \hline 
    \end{tabular} 
\end{table}

The eigenvalues $0$ and $4\pi^2$ are obviously Courant-sharp. To prove Theorem \ref{thmPlejielTorus3D}, we have to show that the eigenvalue $8\pi^2$ is not Courant-sharp, which we will do in the next section.

\section{Courant-type theorem with symmetry}
\label{secCourantSym}

To prove that the eigenvalue $8\pi^2\,$ is not Courant-sharp, we rely on a Courant-type theorem with symmetry, an idea introduced in \cite{Ley96}, and used in \cite{HelHofTer10a,HelfferPersson2015SquareN,HelKiw15Cube} with an objective similar to ours. 

Let $\sigma$ be the isometry of $\mathbb{T}^3$ defined, in the standard coordinates, by
\[\sigma(x,y,z)=(x+1/2\mbox{ mod }1,y+1/2\mbox{ mod }1,z+1/2\mbox{ mod }1)\,.\]
In particular, $\sigma \circ \sigma$ is the identity.
  We call a function $u$ \emph{symmetric} (resp. \emph{antisymmetric}) if $u\circ \sigma=u$ (resp. $u\circ\sigma=-u$). We denote by $L^2_{S,\sigma}(\mathbb{T}^3)$ (resp. $L^2_{A,\sigma}(\mathbb{T}^3)$) the subspace of $L^2(\mathbb{T}^3)$ consisting of all symmetric (resp. antisymmetric) functions. We have the orthogonal decomposition: 
\begin{equation*}
	L^2(\mathbb{T}^3)=L^2_{S,\sigma}(\mathbb{T}^3)\oplus L^2_{A,\sigma}(\mathbb{T}^3).
\end{equation*}
Furthermore, since $\sigma$ is an isometry, $-\Delta_{\mathbb{T}^3}(u\circ \sigma)=\left(-\Delta_{\mathbb{T}^3}u\right)\circ \sigma\,$ for all $u\,$. This implies that both subspaces $L^2_{S,\sigma}(\mathbb{T}^3)$ and $L^2_{A,\sigma}(\mathbb{T}^3)$ are stable under the action of $-\Delta_{\mathbb{T}^3}\,$. Let us now denote by $H_{S,\sigma}$ (resp. $H_{A,\sigma}$) the restriction of $-\Delta_{\mathbb{T}^3}$ to $L^2_{S}(\mathbb{T}^3)$ (resp. $L^2_{A}(\mathbb{T}^3)$), and by $(\lambda_k^{S,\sigma})_{k\ge 1}$ (resp. $(\lambda_k^{A,\sigma})_{k\ge 1}$) the spectrum of $H_{S,\sigma}$ (resp. $H_{A,\sigma}$). The spectrum $(\lambda_k(\mathbb{T}^2))_{k\ge 1}$ is the reunion of $(\lambda_k^{S,\sigma})_{k\ge 1}$ and $(\lambda_k^{A,\sigma})_{k\ge 1}\,$, counted with multiplicities. If $\lambda$ is an eigenvalues of $H_{S,\sigma}\,$, we define the index $\kappa_{S,\sigma}(\lambda)$ by
\[\kappa_{S,\sigma}(\lambda):=\inf\left\{k\,:\, \lambda_k^{S,\sigma}=\lambda\right\}\,.\] 
In the same way, if $\lambda$ is an eigenvalue of $H_{A,\sigma}\,$, we define
\[\kappa_{A,\sigma}(\lambda):=\inf\left\{k\,:\, \lambda_k^{A,\sigma}=\lambda\right\}.\]

Let us now consider $u\,$, a symmetric eigenfunction of $-\Delta_{\mathbb{T}^3}\,$, and $D$, a nodal domain of $u\,$. The set $\sigma(D)$ is also a nodal domain of $u\,$. Either $\sigma(D)=D\,$, in which case we will say that $D$ is symmetric, or we have a pair $\left\{D,\sigma(D)\right\}$ of isometric nodal domains. We denote by $\alpha(u)$ the number of symmetric nodal domains of $u\,$, and by $\beta(u)$ the number of pairs of isometric nodal domains, so that $\nu(u)=\alpha(u)+2\beta(u)$.

If $u$ is an antisymmetric eigenfunction, and if $D$ is a nodal domain of $u\,$, then $\sigma(D)$ is also a nodal domain of $u\,$, distinct from $D$ since the signs of $u$ on $D$ and  $\sigma(D)$ are opposite. Therefore the nodal domains of $u$ can be regrouped into pairs of isometric nodal domains. We denote by $\gamma(u)$ the number of pairs, so that $\nu(u)=2\gamma(u)\,$.

Let us now state a Courant-type theorem with the symmetry $\sigma\,$. The proof is a simple variation of Courant's original argument and will not be given here, for more details see \cite{HelHofTer10a} and references therein.
\begin{thm}
\label{thmCourantSym}
If $\lambda$ is an eigenvalue of $H_{S,\sigma}$ and $u$ an associated symmetric eigenfunction, then
\begin{equation}
	\label{eqIneqSym}
	\alpha(u)+\beta(u)\le \kappa_{S,\sigma}(\lambda)\,.
\end{equation}
If $\lambda$ is an eigenvalue of $H_{A,\sigma}$ and $u$ an associated antisymmetric eigenfunction, then
\begin{equation}
	\label{eqIneqAntiSym}
	\gamma(u)\le \kappa_{A,\sigma}(\lambda)\,.
\end{equation}
\end{thm}

Let us make one additional remark, inspired by the treatment of the cube with a Dirichlet boundary condition in \cite{HelKiw15Cube}. The basis functions generating the vector space $E_{m,n,p}$ (see the proof of Proposition \ref{propIneqN}) are symmetric if the sum $m+n+p$ is even and antisymmetric if it is odd. For any integer triple $(m,n,p)\,$, $m+n+p$ has the same parity as $m^2+n^2+p^2\,$. This implies that the eigenfunctions associated with a given eigenvalue are either all symmetric or all antisymmetric, according to whether $\frac{\lambda}{4\pi^2}$ is even or odd. Equivalently, the spectra of $H_{S,\sigma}$ and $H_{A,\sigma}$ are disjoint.
 
\begin{rem}
	If $\lambda$ is an eigenvalue of $-\Delta_{\mathbb{T}^3}$, it is an eigenvalue of either $H_{S,\sigma}$ or $H_{A,\sigma}$, as seen above. If we consider an associated eigenfunction, the original Courant theorem (Theorem \ref{thmCourant}) still applies, but Theorem \ref{thmCourantSym} can give more information.
\end{rem}

\begin{rem} 
\label{remUpperSym}
If $u$ is a symmetric eigenfunction associated with the eigenvalue $\lambda\,$, then Inequality \eqref{eqIneqSym} and the equality $\nu(u)=\alpha(u)+2\beta(u)$ imply that $\nu(u)\le 2\kappa_{S,\sigma}(\lambda)\,$. If $u$ is an antisymmetric eigenfunction associated with $\lambda\,$, Inequality \eqref{eqIneqAntiSym} is  equivalent to $\nu(u)\le 2\kappa_{A,\sigma}(\lambda)\,$.  We will actually only use the symmetric case in the following.
\end{rem}

Let us now consider the eigenvalue $8\pi^2\,$. It belongs to the spectrum of $H_{S,\sigma}\,$, and $\kappa_{S,\sigma}(8\pi^2)=2\,$. Any eigenfunction associated with $8\pi^2$ is symmetric, and, according to Remark \ref{remUpperSym}, it has at most $4$ nodal domains. This bound is in fact sharp, since for instance the eigenfunction $(x,y,z)\mapsto \cos(2\pi x)\cos(2\pi y)$ has $4$ nodal domains. On the other hand, $\kappa(8\pi^2)=8\,$, so that $8\pi^2\,$, considered as an eigenvalue of $-\Delta_{\mathbb{T}^3}\,$, is not Courant-sharp. This completes the proof of Theorem \ref{thmCourantSym}.

\paragraph{Acknowledgements} The author thanks Bernard Helffer for introducing him to this problem and for numerous discussions and advices. The author also thanks Pierre B{\'e}rard for his careful reading of successive versions of the present work, and for suggesting numerous corrections and improvements, and Susanna Terracini for discussions of isoperimetric inequalities. Most of the present work was done while the author was visiting the \emph{Centre de Recherches Math{\'e}matiques} in Montr{\'e}al, during the \emph{SMS 2015 Summer School: Geometric and Computational Spectral Theory}. The author wishes to acknowledge the financial support of the CRM-ISM. This work was partially supported by the ANR (Agence Nationale de la Recherche), project OPTIFORM n$^\circ$
ANR-12-BS01-0007-02, and by the ERC, project COMPAT n$^\circ$ ERC-2013-ADG.


{\small
	
\begin{thebibliography}{10}
	
	\bibitem{BanBerFaj2015}
	R.~{Band}, M.~{Bersudsky}, and D.~{Fajman}.
	\newblock {A note on Courant sharp eigenvalues of the Neumann right-angled
		isosceles triangle}.
	\newblock {\em ArXiv e-prints}, July 2015.
	\newblock \href {http://arxiv.org/abs/1507.03410} {\path{arXiv:1507.03410}}.
	
	\bibitem{Ber82}
	P.~B{\'e}rard.
	\newblock In\'egalit\'es isop\'erim\'etriques et applications. {D}omaines
	nodaux des fonctions propres.
	\newblock In {\em Goulaouic-{M}eyer-{S}chwartz {S}eminar, 1981/1982}, pages
	Exp. No. XI, 10. \'Ecole Polytech., Palaiseau, 1982.
	
	\bibitem{BerardHelffer2015Triangle}
	P.~{B{\'e}rard} and B.~{Helffer}.
	\newblock {Courant-sharp eigenvalues for the equilateral torus, and for the
		equilateral triangle}.
	\newblock {\em ArXiv e-prints}, February 2015.
	\newblock \href {http://arxiv.org/abs/1503.00117} {\path{arXiv:1503.00117}}.
	
	\bibitem{BerHel15Square}
	P.~B{\'e}rard and B.~Helffer.
	\newblock Dirichlet eigenfunctions of the square membrane: {C}ourant's
	property, and {A}. {S}tern's and {{\AA}}. {P}leijel's analyses.
	\newblock In A.~Baklouti, A.~El~Kacimi, S.~Kallel, and N.~Mir, editors, {\em
		Analysis and Geometry}, volume 127 of {\em Springer Proceedings in
		Mathematics {\&} Statistics}, pages 69--114. Springer International
	Publishing, 2015.
	
	\bibitem{BerMey82}
	P.~B{\'e}rard and D.~Meyer.
	\newblock In\'egalit\'es isop\'erim\'etriques et applications.
	\newblock {\em Ann. Sci. \'Ecole Norm. Sup. (4)}, 15(3):513--541, 1982.
	
	\bibitem{Cha01}
	I.~Chavel.
	\newblock {\em Isoperimetric inequalities}, volume 145 of {\em Cambridge Tracts
		in Mathematics}.
	\newblock Cambridge University Press, Cambridge, 2001.
	
	\bibitem{Cou23Nodal}
	R.~{Courant}.
	\newblock {Ein allgemeiner Satz zur Theorie der Eigenfunktionen
		selbstadjungierter Differentialausdr\"ucke.}
	\newblock {\em {Nachr. Ges. Wiss. G\"ottingen, Math.-Phys. Kl.}}, 1923:81--84,
	1923.
	
	\bibitem{CouHil53}
	R.~Courant and D.~Hilbert.
	\newblock {\em Methods of {M}athematical {P}hysics. {V}ol. {I}}.
	\newblock Interscience Publishers, Inc., New York, N.Y., 1953.
	
	\bibitem{HauPerRomRos04}
	L.~Hauswirth, J.~P{\'e}rez, P.~Romon, and A.~Ros.
	\newblock The periodic isoperimetric problem.
	\newblock {\em Trans. Amer. Math. Soc.}, 356(5):2025--2047 (electronic), 2004.
	
	\bibitem{HelHofTer09}
	B.~Helffer, T.~Hoffmann-Ostenhof, and S.~Terracini.
	\newblock Nodal domains and spectral minimal partitions.
	\newblock {\em Ann. Inst. H. Poincar\'e Anal. Non Lin\'eaire}, 26(1):101--138,
	2009.
	
	\bibitem{HelHofTer10a}
	B.~Helffer, T.~Hoffmann-Ostenhof, and S.~Terracini.
	\newblock On spectral minimal partitions: the case of the sphere.
	\newblock In {\em Around the research of {V}ladimir {M}az'ya. {III}}, volume~13
	of {\em Int. Math. Ser. (N. Y.)}, pages 153--178. Springer, New York, 2010.
	
	\bibitem{HelKiw15Cube}
	B.~{Helffer} and R.~{Kiwan}.
	\newblock {Dirichlet eigenfunctions on the cube, sharpening the Courant nodal
		inequality}.
	\newblock {\em ArXiv e-prints}, June 2015.
	
	\bibitem{HelfferPersson2015SquareN}
	B.~{Helffer} and M.~{Persson Sundqvist}.
	\newblock {Nodal domains in the square---the Neumann case}.
	\newblock {\em Mosc. Math. J.}, 15(3):455--495, jul.--sep. 2015.
	
	\bibitem{HelPer15Balls}
	B.~{Helffer} and M.~{Persson Sundqvist}.
	\newblock {On nodal domains in Euclidean balls}.
	\newblock {\em ArXiv e-prints}, June 2015.
	\newblock \href {http://arxiv.org/abs/1506.04033} {\path{arXiv:1506.04033}}.
	
	\bibitem{HowHutMor99}
	H.~Howards, M.~Hutchings, and F.~Morgan.
	\newblock The isoperimetric problem on surfaces.
	\newblock {\em Amer. Math. Monthly}, 106(5):430--439, 1999.
	
	\bibitem{Len15Torus}
	C.~L{\'e}na.
	\newblock Courant-sharp eigenvalues of a two-dimensional torus.
	\newblock {\em Comptes Rendus Mathematique}, 353(6):535 -- 539, 2015.
	
	\bibitem{Ley96}
	J.~Leydold.
	\newblock On the number of nodal domains of spherical harmonics.
	\newblock {\em Topology}, 35(2):301--321, 1996.
	
	\bibitem{Ple56}
	{\AA}.~Pleijel.
	\newblock Remarks on {C}ourant's nodal line theorem.
	\newblock {\em Comm. Pure Appl. Math.}, 9:543--550, 1956.
	
	\bibitem{Pol09}
	I.~Polterovich.
	\newblock Pleijel's nodal domain theorem for free membranes.
	\newblock {\em Proc. Amer. Math. Soc.}, 137(3):1021--1024, 2009.
	
	\bibitem{Ros05}
	A.~Ros.
	\newblock The isoperimetric problem.
	\newblock In {\em Global theory of minimal surfaces}, volume~2 of {\em Clay
		Math. Proc.}, pages 175--209. Amer. Math. Soc., Providence, RI, 2005.
	
	\bibitem{TotZel09}
	J.~A. Toth and S.~Zelditch.
	\newblock Counting nodal lines which touch the boundary of an analytic domain.
	\newblock {\em J. Differential Geom.}, 81(3):649--686, 2009.
	
\end{thebibliography}

}
\end{document}